\newfont{\Bbb}{msbm10 scaled\magstephalf}
\documentclass[12pt,reqno]{amsart}
\usepackage{amsfonts}
\usepackage{amsmath, amsthm, amscd, amsfonts}
\usepackage{amssymb}
\usepackage{amsmath}
\usepackage{xcolor}
\usepackage{amscd}
\usepackage{graphicx}
 \usepackage{epstopdf}
\allowdisplaybreaks[4]
 \newtheorem{thm}{Theorem}[section]
 \newtheorem{cor}[thm]{Corollary}
 \newtheorem{lem}[thm]{Lemma}
 
 \theoremstyle{definition}
 \newtheorem{defn}[thm]{Definition}
\theoremstyle{remark}
 \newtheorem{rem}[thm]{Remark}
 \newtheorem{exm}[thm]{Example}
 \numberwithin{equation}{section}
\newcommand{\pf}{\begin{proof}}
\newcommand{\zb}{\end{proof}}
\newcommand{\la}{\langle}
\newcommand{\ra}{\rangle}
\newcommand{\ol}{\overline}
\newcommand{\ma}{\mathcal}
\newcommand{\Ker}{\mathop{\rm Ker}\nolimits}

\def\DD{{\mathbb D}}

\begin{document}

\title[nearly invariant subspace]{nearly invariant subspaces for operators in Hilbert spaces}
\author[Y. Liang]{Yuxia Liang}
\address{Yuxia Liang \newline School of Mathematical Sciences,
Tianjin Normal University, Tianjin 300387, P.R. China.} \email{liangyx1986@126.com}
\author[J. R. Partington]{Jonathan R. Partington}
\address{Jonathan R. Partington \newline School of Mathematics,
  University of Leeds, Leeds LS2 9JT, United Kingdom.}
 \email{J.R.Partington@leeds.ac.uk}

\begin{abstract} For a shift operator $T$ with finite multiplicity acting on a separable infinite dimensional Hilbert space we represent its nearly $T^{-1}$ invariant subspaces in Hilbert space in terms of invariant subspaces under the backward shift. Going further, given any finite Blaschke product $B$, we give a description of the nearly $T_{B}^{-1}$ invariant subspaces for the operator $T_B$ of multiplication by $B$
 in a scale of Dirichlet-type spaces.
\end{abstract}

\maketitle

\section{Introduction}
Given $\alpha$ a real number, the Dirichlet-type space $\ma{D}_\alpha(\mathbb{D})$ consists of all analytic functions $f(z)=\sum_{k=0}^\infty a_k z^k$ in $\mathbb{D}$ such that its norm
$$\|f\|_\alpha:=\left(\sum_{k=0}^\infty|a_k|^2(k+1)^\alpha
\right)^{1/2}<+\infty .$$ If $\alpha=-1$, $\ma{D}_{-1}=A^2(\mathbb{D})$ the classical Bergman space, for $\alpha=0,$ $\ma{D}_0=H^2(\mathbb{D})$ the Hardy space, and for $\alpha=1,$ $\ma{D}_{1}=\ma{D}(\mathbb{D})$ the classical Dirichlet space, systematically investigated in the book \cite{EKMR}. These are particular instances of separable infinite dimensional Hilbert spaces, to be denoted by $\ma{H}$ in this paper. We let $\ma{B}(\ma{H})$ denote the collection of all bounded linear operators acting on $\ma{H}.$

The notations $\mathbb{N}_0$ and $\mathbb{N}$ denote the set of all nonnegative integers and positive integers, respectively. Here we recall the $\mathbb{C}^l$-vector-valued Hardy space  $H^2(\mathbb{D},\mathbb{C}^l)$ consists of all analytic $F: \mathbb{D}\rightarrow \mathbb{C}^l$ such that the norm
$$\|F\|=\left(\sup\limits_{0<r<1}\frac{1}{2\pi} \int_0^{2\pi} \|F(re^{iw})\|^2 dw\right)^{\frac{1}{2}}<\infty.$$
Writing $F=[f_1,f_2,\cdots, f_l]$ with $f_i: \mathbb{D}\rightarrow \mathbb{C}$, it is clear that
$F\in H^2(\mathbb{D},\mathbb{C}^l)$ if and only if $f_i\in H^2(\mathbb{D})$ for $i=1,2,\cdots,l$ with $l\in \mathbb{N}.$

Inner functions play important role in describing the invariant subspaces of the unilateral shift $Sf(z)=zf(z)$ (multiplication by the independent variable) on $H^2(\mathbb{D})$. Beurling's Theorem states that a nontrivial closed subspace $\ma{M}\subset H^2(\mathbb{D})$ satisfies $S\ma{M}\subset \ma{M}$ if and only if $\ma{M}=\theta H^2(\mathbb{D})$ with $\theta$ is inner. The simplest nontrivial inner function is an automorphism of $\mathbb{D}$ mapping $\mathbb{T}$ onto $\mathbb{T}.$ More generally, if $\{a_n\}_{n\geq 1}$ is a sequence of points in $\mathbb{D}\setminus\{0\}$ satisfying the Blaschke condition $\sum_{n=1}^\infty (1-|a_n|)<\infty,$ then we can construct the corresponding Blaschke product $$B(z):=z^m \prod_{n=1}^\infty \frac{|a_n|}{a_n}\frac{a_n-z}{1-\ol{a_n}z},\;m\in \mathbb{N}_0.$$

The Toeplitz operator $T_g$ on $H^2(\mathbb{D})$ is defined by
\begin{eqnarray*}T_g f=P(gf)\;\mbox{with}\; g\in L^\infty(\mathbb{T})\label{Toeplitz}\end{eqnarray*} and $P$ is the orthogonal projection from $ L^2(\mathbb{T})$ on $H^2(\mathbb{D})$. It is well-known that the kernel of $T_{\ol{\theta}}$ on $H^2(\mathbb{D})$ is a model space $K_{\theta}:=H^2\ominus \theta H^2$ with $\theta$ an inner function (cf. \cite{GMR,Ha,Sa2}).

Given a Blaschke product $B$, the Wold Decomposition Theorem implies that every $f \in H^2(\DD)$ has an expression
\begin{eqnarray*}
f(z)=\sum_{k=0}^\infty B^k(z)h_k(z)\label{Bmm}
\end{eqnarray*}
with $h_k$ are functions in $K_{B}=H^2\ominus B H^2$. An analogous theorem for Dirichlet-type space $\ma{D}_\alpha(\mathbb{D})$ has been proved
as follows.\vspace{1mm}

\begin{thm}\cite[Theorem 3.1]{GPS}, \cite[Theorem 2.1]{CGP}  Let $\alpha\in [-1,1]$ and $B$ a finite Blaschke product. Then $f\in \ma{D}_\alpha(\mathbb{D})$ if and only if $f=\sum_{k=0}^\infty B^k h_k$ (convergence in $\ma{D}_\alpha$ norm) with $h_k\in K_B$ and \begin{eqnarray}\sum_{k=0}^\infty(k+1)^\alpha \|h_k\|_{H^2}^2<\infty.\label{normB}\end{eqnarray}\end{thm}

Since $K_B$ is finite-dimensional, we may take other (equivalent) norms here, such as $\|h_k\|_{\ma{D}_\alpha}$.\\

A concept commonly appearing in operator theory and complex
function theory is that of near invariance, which arises
in the investigations of (almost) invariant subspaces. At the beginning, nearly $S^*$ invariant subspaces of $H^2(\mathbb{D})$ were introduced by Hayashi \cite{Ha}, Hitt  \cite{hitt}, and then Sarason \cite{Sa1} in the context of kernels of Toeplitz operators. There are also many other contributions related with this topic; for example the case of backwards shifts on vector-valued Hardy spaces was analysed in \cite{CCP}. The interested reader can also refer to \cite{CaP1,CaP,LP} and the references therein. Roughly speaking, a subspace $\ma{M}\subset \ma{H}$ is said to be nearly $S^*$ invariant if the zeros of functions in $\ma{M}$ can be divided out without leaving the space. The following definition presents a nearly $T^{-1}$ invariant subspace for any left invertible $T\in \ma{B}(\ma{H})$.

\begin{defn}\label{defn T}
Let $\ma{H}$ be a separable infinite dimensional Hilbert space and $T\in \ma{B}(\ma{H})$ be left invertible.
Then a subspace $\ma{M}\subset \ma{H}$ is said to be nearly $T^{-1}$ invariant if for every $g\in \ma{H}$ such that $Tg\in \ma{M},$   it holds that  $g\in \ma{M}$.
\end{defn}

A shift operator acting on a separable infinite dimensional Hilbert space is the direct generalization of the unilateral shift $S$ and multiplication operator $T_B$ on $H^2(\mathbb{D},\mathbb{C}^l)$. It is isometric and left invertible,
and was abstractly defined in \cite[Chapter 1]{RR} as below.

\begin{defn}\label{defn shift} An operator $T\in \ma{B}(\ma{H})$ is a shift operator if $T$ is an isometry  and
$\|T^{*n} f\|\rightarrow 0$ for all $f\in \ma{H}$ as $n\rightarrow \infty.$  \end{defn}

There are also other equivalent descriptions for a shift operator. For example, an isometry $T\in \ma{B}(\ma{H})$ is a \emph{shift operator} if and only if $T$ is \emph{pure}. Here an isometry $T$  on $\ma{H}$ is \emph{pure} whenever $\bigcap_{n\geq 0} T^n \ma{H}=\{0\}.$  Furthermore,   a pure isometry $T\in \ma{B}(\ma{H})$ has \emph{multiplicity} $m$ if the dimension of the subspace $\mathcal{K}:=\ma{H}\ominus T\ma{H}=\Ker T^*$ is $m$.

Besides, an operator $A\in \ma{B}(\ma{H})$ is $T$-inner if $A$ is analytic (that is, $AT=TA$) and partially isometric. Based on the concept of $T$-inner operator, the Beurling-Lax Theorem for invariant subspaces under a shift operator $T\in \ma{B}(\ma{H})$ was given in \cite[Section 1.12]{RR}, as follows.

\begin{thm} \label{thm A}A subspace $\ma{F}$ of $\ma{H}$ is invariant under the shift operator $T\in \ma{B}(\ma{H})$ if and only if $\ma{F}=A\ma{H}$ for some $T$-inner operator $A$ on $\ma{H}.$\end{thm} It is natural to ask how to give an expression for nearly $T^{-1}$ invariant subspaces in $\ma{H}$. Motivated by this question, we organize the rest of the paper  as follows. In Section 2, using the formulae of $S^*$ invariant subspaces in vector-valued Hardy space, we present a characterization for nearly $T^{-1}$ invariant subspaces when $T\in \ma{B}(\ma{H})$ is a shift operator with \emph{multiplicity} $m$. Going beyond this, noting that a finite Blaschke product $B$ is a multiplier of Dirichlet-type spaces, we give some descriptions for nearly $T_B^{-1}$ invariant subspaces in Section 3, extending considerably some work of Erard \cite{Er}.

\section{nearly $T^{-1}$ invariant subspaces}

In this section, we always suppose $T\in \ma{B}(\ma{H})$ is a shift operator with multiplicity $m$. This gives \begin{eqnarray}1\leq l:=\mbox{dim} [\ma{M}\ominus (\ma{M} \cap T\ma{H})]\leq m\label{l}\end{eqnarray} for every nonzero nearly $T^{-1}$ invariant subspace $\ma{M}\subset\ma{H}$. Denote an orthonormal basis of $\ma{K}:=\ma{H}\ominus T\ma{H}$ by $e_1,\cdots, e_m$. And let $\delta_j^m=(0,\cdots,0, 1,0,\cdots,0)$ with $1$ in the $j$-th place be an orthonormal basis of $K:=H^2(\mathbb{D}, \mathbb{C}^m)\ominus zH^2(\mathbb{D}, \mathbb{C}^m)$ for $j=1, 2, \cdots, m.$ Based on the following two orthogonal decompositions
\begin{eqnarray*} \ma{H}=\bigoplus_{i=0}^\infty  T^i\ma{K}\;\;\mbox{and}\;\;H^2(\mathbb{D}, \mathbb{C}^m)=\bigoplus_{i=0}^\infty  z^i K,\end{eqnarray*}
 there exists a unitary mapping $U:\;\ma{H}\rightarrow H^2(\mathbb{D}, \mathbb{C}^m)$ defined by
\begin{eqnarray}
U(T^i e_j)=z^i\delta_j^m.\label{eq:defu}\end{eqnarray}

So the following commutative diagram \eqref{commute} holds for the unilateral shift $S:\;H^2(\mathbb{D}, \mathbb{C}^m)\rightarrow H^2(\mathbb{D}, \mathbb{C}^m)$ and the shift operator $T: \;\ma{H} \rightarrow \ma{H}$ with multiplicity $m$.
  \begin{eqnarray}\begin{CD}
\ma{H} @>T>> \ma{H}\\
@VV U V @VV U V\\
H^2(\mathbb{D}, \mathbb{C}^m) @>S>> H^2(\mathbb{D}, \mathbb{C}^m),
\end{CD}\label{commute}\end{eqnarray}
which implies the following equations
\begin{eqnarray}&&S^nU=UT^n\;\mbox{for}\;n\in \mathbb{N}_0.\label{nn}\end{eqnarray}
 Using the unitary mapping $U: \ma{H}\rightarrow H^2(\mathbb{D}, \mathbb{C}^m)$, the following lemma holds, where the superscript `t' means the transpose of a matrix.

\begin{lem}\label{lem FG} Let $\ma{M}\subset \ma{H}$ be a nonzero nearly $T^{-1}$ invariant subspace and  $G_0:=[g_1, g_2, \cdots, g_l]^t$ be a matrix containing an orthonormal basis $(g_i)_{i\in\{1,\cdots,l\}}$ of $\ma{M}\ominus (\ma{M}\cap T\ma{H})$. Then

\begin{eqnarray}F_0:=[Ug_1, Ug_2, \cdots, Ug_l]^t\label{F0}\end{eqnarray}
is a matrix containing an orthonormal basis $(Ug_i)_{i\in\{1,\cdots,l\}}$ of $U\ma{M}\ominus (U\ma{M}\cap zH^2(\mathbb{D},\mathbb{C}^m)).$
\end{lem}

The lemma below shows the link of nearly invariant subspaces between similar operators.

\begin{lem}\label{lem similar} Suppose $\ma{H}_1$ and $\ma{H}_2$ are two Hilbert spaces and $T_1\in \ma{B}(\ma{H}_1)$ and $T_2\in \ma{B}(\ma{H}_2)$ are two left invertible operators. Assume there exists  an invertible operator $V: \ma{H}_1\rightarrow \ma{H}_2$ so that $T_2= V T_1 V^{-1}$. Let $\ma{M}$ be a nearly $T_1^{-1}$ invariant subspace in $\ma{H}_1$, then $V(\ma{M})$ is a nearly $T_2^{-1}$ invariant subspace in $\ma{H}_2.$ \end{lem}
\begin{proof}  For any $h\in \ma{H}_2,$ if $T_2 h\in V(\ma{M}),$ we need to show $h\in V(\ma{M}).$ Since $T_2 h=VT_1 V^{-1}h \in V(\ma{M}),$ it follows that $ T_1 V^{-1} h\in \ma{M},$ so that $V^{-1} h\in \ma{M}$. This means $h\in V(\ma{M}),$ ending the proof.\end{proof}

\begin{lem} \label{lem UT} Suppose that $T: \mathcal{H}\rightarrow \mathcal{H}$ is a shift operator,
and let $U$ be as \eqref{eq:defu}. Then
\begin{eqnarray}
U^*[(Ug)h]=h(T)g\label{fug},
\end{eqnarray}
for any $g\in \ma{H},\;h\in H^2(\mathbb{D})$ such that $(Ug)h\in H^2(\mathbb{D},\mathbb{C}^m)$.
\end{lem}

\begin{proof}From the commutative diagram \eqref{commute}, we have $Ug\in H^2(\mathbb{D},\mathbb{C}^m)$ and $(Ug)z^n=S^n (Ug),$  and then \eqref{nn}  implies that
\begin{eqnarray*} U^*[(Ug)z^n]=U^*S^n(Ug)=T^nU^*Ug=T^ng, \;\mbox{for}\;n\in \mathbb{N}_0. \end{eqnarray*}

For any polynomial $p(z)=\sum_{k=0}^n a_kz^k \in H^2(\mathbb{D}),$ since $U^*$ is a linear operator, we have
\begin{eqnarray*} U^*[(Ug)p(z)]=\sum_{k=0}^n a_k T^k g=p(T)g,\end{eqnarray*} where the operator $p(T)=\sum_{k=0}^n a_k T^k .$ So for any $h(z)=\sum_{k=0}^\infty h_kz^k\in H^2(\mathbb{D})$ such that $(Ug)h\in H^2(\mathbb{D},\mathbb{C}^m)$, there exists a sequence of polynomials $q_n(z)=\sum_{k=0}^nh_kz^k\in H^2(\mathbb{D})$ such that $q_n\rightarrow h$ in $H^2(\mathbb{D})$, as $n\rightarrow \infty.$  Then we deduce that
\begin{eqnarray*} U^*[(Ug)h(z)]=U^*[(Ug)(\lim\limits_{n\rightarrow \infty}q_n(z))]=\lim\limits_{n\rightarrow \infty}q_n(T)g=h(T)g,\end{eqnarray*} with the operator $h(T)=\sum_{k=0}^\infty h_kT^k$. \end{proof}
We are now in a position to state a theorem for nearly $T^{-1}$ invariant subspaces in $\ma{H}$. Recall that $\Phi\in H^\infty(\mathbb{D}, \ma{L}(\mathbb{C}^{l'}, \mathbb{C}^l))$ is an operator-valued inner function if $\Phi(e^{iw})$ is an isometry a.e. on $\mathbb{T}.$

\begin{thm}\label{thm nearlyTC}
Suppose $T: \mathcal{H}\rightarrow \mathcal{H}$ is a shift operator with  multiplicity
$m$ and $\ma{M}\subset \ma{H}$ is a nonzero nearly $T^{-1}$ invariant subspace. Let $G_0:=\begin{bmatrix}g_1, g_2, \cdots,  g_l\end{bmatrix}^t$ be a matrix containing an orthonormal basis $(g_i)_{i\in\{1,\cdots,l\}}$ of $\ma{M}\ominus(\ma{M}\cap T \ma{H})$. Then there exist a nonnegative integer $l'\leq l$ and an operator-valued inner function $\Phi$ belonging to $H^\infty(\mathbb{D}, \ma{L}(\mathbb{C}^{l'}, \mathbb{C}^l))$, unique up to unitary equivalence, such that
\begin{eqnarray*}\ma{M}=\{f\in \ma{H}:\;\exists h\in H^2(\mathbb{D},\mathbb{C}^l)\ominus \Phi H^2(\mathbb{D},\mathbb{C}^{l'}),\;f=h(T) G_0\}. \label{c1} \end{eqnarray*}Besides, there is an isometric mapping
\begin{eqnarray}Q:\;\ma{M}\rightarrow H^2(\mathbb{D},\mathbb{C}^l)\ominus \Phi H^2(\mathbb{D},\mathbb{C}^{l'})\;\mbox{given by}\; Q(f)=h.\label{Q}\end{eqnarray}
\end{thm}

\begin{proof} From Lemma \ref{lem similar} and the commutative diagram \eqref{commute}, it follows $U\ma{M}$ is a nearly $S^*$ invariant subspace in $H^2(\mathbb{D}, \mathbb{C}^m).$
From \cite[Theorem 4.4]{CCP}, there exists an isometric mapping
\begin{eqnarray}J:\;U\ma{M}\rightarrow \ma{F}'\;\;\mbox{given by }\;J(h F_0)=h,\label{J}\end{eqnarray}
with a subspace $$\ma{F}':=\{h\in H^2(\mathbb{D}, \mathbb{C}^l):\; \exists\; Uf\in U \ma{M},\;\;Uf= h F_0 \}$$ and $F_0$ given in \eqref{F0}.  Moreover, $\ma{F}'$ is an $S^*$ invariant subspace in $H^2(\mathbb{D},\mathbb{C}^l).$ The Beurling-Lax Theorem on $H^2(\mathbb{D}, \mathbb{C}^l)$ implies there exist a nonnegative integer $l'\leq l$ and an inner function $\Phi\in H^\infty(\mathbb{D}, \ma{L}(\mathbb{C}^{l'}, \mathbb{C}^l)),$ unique to unitary equivalence, such that
$$\ma{F}'=H^2(\mathbb{D},\mathbb{C}^l)\ominus \Phi H^2(\mathbb{D},\mathbb{C}^{l'}).$$

Given any $f\in \ma{M},$ the isometric mapping \eqref{J} implies there exists $h=[h_1, \cdots, h_l]\in \ma{F}'$ such that
 \begin{eqnarray*}Uf=hF_0=[h_1, \cdots, h_l][Ug_1, \cdots, Ug_l]^t
=\sum_{i=1}^l(Ug_i)h_i,\end{eqnarray*} with $\|f\|=\|Uf\|=\|h\|.$
Then the formula \eqref{fug} gives
 \begin{eqnarray*}f=\sum_{i=1}^l U^*[(U g_i)h_i]
 =\sum_{i=1}^l h_i(T) g_i=h(T) G_0,\label{fg}\end{eqnarray*} with $h(T)=[h_1(T),\cdots, h_l(T)].$ We therefore have
\begin{eqnarray*}\ma{M}=\{f\in \ma{H}:\;\exists h\in H^2(\mathbb{D},\mathbb{C}^l)\ominus \Phi H^2(\mathbb{D},\mathbb{C}^{l'}),\;f=h(T) G_0\},\end{eqnarray*}
and there is an isometric mapping $Q: \ma{M}\rightarrow H^2(\mathbb{D},\mathbb{C}^l)\ominus \Phi H^2(\mathbb{D},\mathbb{C}^{l'})$ given by $Q=JU$ satisfying \eqref{Q}.
\end{proof}

\begin{rem}Let $$ \ma{M}'=U^*\ma{F}'=U^*[H^2(\mathbb{D},\mathbb{C}^l)\ominus \Phi H^2(\mathbb{D},\mathbb{C}^{l'})]\subset \ma{H}.$$
Then $\ma{M}'$ is a $T^*$ invariant subspace and there also exists an isometric mapping $\widetilde{Q}=U^*JU: \ma{M}\rightarrow \ma{M}'$ defined by $f\mapsto U^*h.$ \end{rem}

Given a degree-$m$ Blaschke product $B$, $T_{B}: H^2(\mathbb{D})\rightarrow H^2(\mathbb{D})$ is a   shift operator with  multiplicity $m$, so we can deduce a corollary.

\begin{cor}\label{cor nearlyBm} Let $\ma{M}\subset H^2(\mathbb{D})$ is a nonzero nearly $T_B^{-1}$ invariant subspace with a degree-$m$ Blaschke product $B$. Let the matrix $G_0(z):=[g_1(z), g_2(z), \cdots,  g_l(z)]^t$ contain an orthonormal basis $(g_i(z))_{i\in\{1,\cdots,l\}}$ of $\ma{M}\ominus(\ma{M}\cap B H^2(\mathbb{D}))$.
Then there exist a nonnegative integer $l'\leq l$ and an operator-valued inner function $\Phi\in H^\infty(\mathbb{D}, \ma{L}(\mathbb{C}^{l'}, \mathbb{C}^l))$, unique up to unitary equivalence, such that
\begin{eqnarray*} \ma{M} =\{f\in H^2(\mathbb{D}): \exists h\in H^2(\mathbb{D},\mathbb{C}^l)\ominus \Phi H^2(\mathbb{D},\mathbb{C}^{l'}), f(z)=h(T_B)G_0(z)\}.\quad\quad\label{Mb1}\end{eqnarray*}
\end{cor}

Here we show an example to illustrate Corollary \ref{cor nearlyBm}.

\begin{exm} Given $a\in \mathbb{D}\setminus\{0\}$ and some $m\in \mathbb{N}_0,$ denote a subspace $$\ma{M}=\varphi_a(z)\cdot\left(\bigvee\{z^{2k}:\;k\in \mathbb{N}\} \oplus \bigvee\{z, z^3, \cdots, z^{2m+1}\}\right),$$ which is nearly $T_{{z^2}}^{-1}$ invariant in $H^2(\mathbb{D})$. It holds that
$$\ma{M}\ominus(\ma{M}\cap z^2 \ma{M})=\la g_1(z), g_2(z)\ra$$ with the function matrix $G_0(z):=[g_1(z),g_2(z)]^t=\varphi_a(z)\cdot \begin{bmatrix}1, z \end{bmatrix}^t.$

For any $f\in \ma{M}$, it turns out
\begin{eqnarray*}f(z)= \left[\sum_{i=0}^\infty a_{i1} z^{2i}, \sum_{i=0}^m a_{i,2}z^{2i}\right]G_0(z),  \end{eqnarray*} where $(a_{i,j})_{i,j}$ satisfies
\begin{eqnarray*}[a_{i,1}, a_{i,2}] \in\left\{\begin{array}{ll}
      \mathbb{C}\times \mathbb{C}, &i=0,1,\cdots,m, \\
      \mathbb{C}\times \{0\},  &i\geq m+1.\end{array}                                  \right. \end{eqnarray*}
The formula in Corollary \ref{cor nearlyBm} together the above facts imply
\begin{eqnarray*}\ma{M}=\left\{f\in H^2(\mathbb{D}): \exists h\in H^2(\mathbb{D}, \mathbb{C}^2)\ominus \Phi(z) H^2(\mathbb{D}), f(z)=h(T_{z^2})G_0(z)\right\},\end{eqnarray*}
with an operator-valued inner function $\Phi(z)=z^{m+1}(0,1)\in \mathbb{C}^2$.
\end{exm}

\section{Nearly $T_{B}^{-1}$ invariant subspaces in $\ma{D}_\alpha$ }

In this section we address the more difficult question of near invariance for the operator $T_B: \ma{D}_\alpha(\mathbb{D})\rightarrow \ma{D}_\alpha(\mathbb{D})$ with a degree-$m$ Blaschke product  $B$,  which is not isometric but simply
bounded below, extending the methods of Erard \cite{Er}.

It is known that the special multiplication operator $T_B$ is always bounded on Dirichlet-type space $\ma{D}_\alpha: =\ma{D}_\alpha(\mathbb{D})$ for any finite Blaschke product $B$. The study of multiplication invariant subspaces of Hardy spaces can be
traced  back to Lance and Stessin's work \cite{LS}; they described closed subspaces of the Hardy spaces $H^p(\mathbb{D})$ which are inner-invariant. In 2004, Erard investigated the nearly invariant subspaces related to lower bounded multiplication operator $M_u$ on a Hilbert space $\ma{H}$ in \cite{Er}. There are four conditions on the pairs $(\ma{H}, u)$ as below:\vspace{1mm}

$(i)$\; $\ma{H}$ is a Hilbert space and a linear submanifold of
\begin{center}
$\ma{O}$($\ma{W}$)$ :=\{f:\;\ma{W}\rightarrow \mathbb{C}\; |\; f\;\mbox{is analytic}\},$
\end{center}
where $\ma{W}$ is an open subset of $\mathbb{C}^d$ ($d\in \mathbb{N}$),

$(ii)$\;$u\in \ma{O}$($\ma{W}$) satisfies $uh\in \ma{H}$ for all $h\in \ma{ H}$,

$(iii)$\;for all $w \in \ma{W}$, the evaluation $\ma{H}\rightarrow  \mathbb{C}$, $h\rightarrow h(w)$ is continuous,

$(iv)$\;there exists $c > 0$ such that for all $h \in \ma{H}$,\;$c\|h\|_{\ma{H}}\leq \|uh\|_{\ma{H}}$. \vspace{1mm}

In the sequel, a ``subspace" means a closed linear subspace, and a ``linear manifold'' is an algebraic subspace that is not necessarily closed.

For the above $(\ma{H},u),$ the \emph{lower bound} of $M_u$ relative to the norm $\|\cdot\|_{\ma{H}}$ is defined by \begin{eqnarray}\gamma_{\ma{H},M_u}=\sup\{c>0:\; \forall\;h\in \ma{H},\;c\|h\|_{\ma{H}}\leq\|uh\|_{\ma{H}}\}\in ]0,\infty[.\label{gamma}\end{eqnarray}

Erard  gave the definition of ``nearly invariant under division by $u$", which is same as ``nearly $M_u^{-1}$ invariant", a special case of Definition \ref{defn T}. Considering the importance of the backward shift $M_z$, Erard proved the following theorem on nearly $S^*$ invariant subspaces in $\ma{H},$ under the assumption that $M_z: \ma{H}\rightarrow \ma{H}$ is bounded below.

\begin{thm} \cite[Theorem 5.1]{Er} Assume that $\ma{H}$ satisfies $(i)$-$(iv)$ with $u(z)=z$, and
$$dim (\ma{H}\ominus M_z \ma{H})=1\;\mbox{and}\; \|h\|_{\ma{H}}\leq \|M_z h\|_{\ma{H}}\;\mbox{for all}\;h\in \ma{H}.$$ Assume also that there exists $f\in \ma{H}$ with $f(0)\neq 0$. Let $\ma{M}$ be a nonzero subspace of $\ma{H}$ which is nearly invariant under the backward shift $M_z$. Let $g$ be any unit vector of $\ma{M}\ominus(\ma{M}\cap M_z \ma{H}).$ Then there exists a linear submanifold $\ma{N}$ of $H^2(\mathbb{D})$ such that $\ma{M}=g\ma{N}$ and for all $h\in \ma{M}$, we have $$\|h\|_{\ma{H}}\geq \left\|\frac{h}{g}\right\|_{H^2(\mathbb{D})}.$$ Besides, $\ma{N}$ is invariant under the backward shift and $g(0)\neq 0.$\end{thm}

We note the operator $T_B: \ma{D}_\alpha\rightarrow \ma{D}_\alpha$ is more general than $M_z: H^2(\mathbb{D})\rightarrow H^2(\mathbb{D})$, so we seek characterizations for nearly $T_{B}^{-1}$ invariant subspaces in $\ma{D}_\alpha$ with $\alpha\in[-1, 1]$ and a degree-$m$ Blaschke product $B$. The following lemma is the factorization of elements in a nearly $T^{-1}$ invariant subspace with a bounded below $T\in \ma{B}(\ma{H})$. (We use the notation $P^\ma{N}$ for the orthogonal projection onto a subspace $\ma{N}$.)

\begin{lem} \cite[Lemma 2.1]{Er}\label{lem decomp} Let $\ma{H}$ be a Hilbert space, $T: \ma{H}\rightarrow \ma{H}$ be a bounded operator such that for all $h\in \ma{H},$ $\|h\|_{\ma{H}}\leq \|Th\|_{\ma{H}}$ and $\ma{M}$ be a nearly $T^{-1}$ invariant subspace of $\ma{H}$. Set
\begin{eqnarray}R=(T^* T)^{-1} T^* P^{\ma{M}\cap T\ma{H}}\;\mbox{and}\;\; Q=P^{\ma{M}\ominus(\ma{M}\cap T\ma{H})}.\label{RQ}\end{eqnarray} Then for all $h\in \ma{M}$ and $p\in \mathbb{N}_0$, we have
\begin{eqnarray} h=\sum_{k=0}^p T^k QR^kh +T^{p+1} R^{p+1}h.\label{hfor}\end{eqnarray}
\end{lem}
In this section, the range of the symbol $l$ in \eqref{l} is also true for $T=T_B$, a nonzero nearly $T_B^{-1}$ invariant subspace $\ma{M}$ of $\ma{H}=\ma{D}_\alpha$ with $\alpha\in[-1, 1]$ and a degree-$m$ Blaschke product $B$. In the sequel, we will endow the space $\ma{D}_\alpha$ with two different equivalent norms according to the cases $\alpha\in [0,1]$ and $\alpha\in [-1,0)$, so we divide the discussion into two subsections.

\subsection{Nearly $T_{B}^{-1}$ invariant subspaces in $\ma{D}_\alpha$ with $\alpha\in [0,1]$}

 In this subsection, $\ma{D}_\alpha$ is endowed with an equivalent norm as in \eqref{normB} denoted by $\| \cdot\|_1$, that is,
  \begin{eqnarray}
  \|f\|_{1}^2:=\sum_{k=0}^\infty(k+1)^\alpha \|h_k\|_{H^2}^2  \label{eq:d01norm}
  \end{eqnarray}
   for any  $f=\sum_{k=0}^\infty B^k h_k$ with $h_k\in K_{B}$.
Then it holds that
$$\|T_{B}f\|_1^2=\|Bf\|_{1}^2=\sum_{k=0}^\infty(k+2)^\alpha \|h_k\|_{H^2}^2\geq \|f\|_{1}^2,$$
which implies the operator $T_{B}: \ma{D}_\alpha \rightarrow \ma{D}_\alpha$ is lower bounded. This
confirms that $(\ma{D}_\alpha, T_{B})$ satisfies   \emph{Conditions} $(i)$-$(iv)$ and the lower bound of $T_{B}$ relative the norm $\|\cdot\|_{1}$  is $\gamma_1:=1.$ And  it is also true that
$B^{-1}(D(0,1))=B^{-1}(\mathbb{D})=\mathbb{D}.$ The above facts together with  $$\bigcap_{n\in \mathbb{N}} B^n \ma{D}_\alpha=\{0\}\;\;\mbox{on}\;\mathbb{D}$$  imply the following lemma, which can be deduced from \cite[Theorem 3.2]{Er} with $\ma{H}=\ma{D}_\alpha,\;u=B$, $\gamma=\gamma_1:=1$ and the index set $I=\{1, 2,\cdots, l\}.$

\begin{lem}\label{lem alpha1}For $\alpha\in[0, 1]$, let $\ma{M}$ be a nonzero nearly $T_{B}^{-1}$ invariant subspace of $\ma{D}_\alpha$ endowed with the norm  $\|\cdot\|_1$ in \eqref{eq:d01norm} and $(g_i)_{i\in\{1,\cdots,l\}}$ be a hilbertian basis of $\ma{M}\ominus (\ma{M}\cap T_{B}\ma{D}_\alpha)$.  Then for all $h\in \ma{M},$ there exists $(q_i)_{i\in\{1,\cdots, l\}}$ in $\ma{O}(\mathbb{D})$ such that
\begin{eqnarray}h=\sum_{i=1}^l g_i q_i\;\mbox{on}\;\mathbb{D},\label{halpha1} \end{eqnarray}
and for all $i\in \{1,\cdots, l\},$ there exists a sequence $\{c_{ki}\}_{k\in \mathbb{N}}\in \mathbb{C}^{\mathbb{N}}$ with
\begin{eqnarray}&&q_i=\sum_{k=0}^\infty c_{ki}  B ^k,\label{qi1}\\ &&\sum_{i=1}^l\sum_{k=0}^\infty |c_{ki}|^2\leq \|h\|_{1}^2.\label{cki1} \end{eqnarray}\end{lem}

Now we are ready to state a theorem on nearly $T_{B}^{-1}$ invariant subspaces in $\ma{D}_\alpha$ with $\alpha\in [0,1].$

 \begin{thm} \label{thm Dalpha} For $\alpha\in[0, 1]$, let $\ma{M}$ be a nonzero nearly $T_{B}^{-1}$ invariant subspace of $\ma{D}_\alpha$ endowed with the norm $\|\cdot\|_1$ in \eqref{eq:d01norm} and $G_0:=[g_1, g_2,\cdots, g_l]^t$ be a matrix containing an orthonormal basis $(g_i)_{i\in\{1,\cdots,l\}}$ of $\ma{M}\ominus (\ma{M}\cap T_{B}\ma{D}_\alpha)$.  Then there exists a linear submanifold $\ma{N}\subset H^2(\mathbb{D}, \mathbb{C}^l)$  such that $\ma{M}=\ma{N}G_0$ and for all $h\in \ma{M}$, there exists $q\in \ma{N}$ such that $h= q G_0$ and $$\|h\|_{1}\geq \|q\|_{H^2(\mathbb{D}, \mathbb{C}^l)}.$$
 Moreover, $\ma{N}$ is invariant under $T_{\ol{B}}$.\end{thm}

\begin{proof} For $\alpha\in[0, 1]$, the equation \eqref{halpha1} implies that every $h\in \ma{M}$ has the form
\begin{eqnarray}h=\sum_{i=1}^l g_i q_i= q G_0 \;\mbox{on}\;\mathbb{D},\label{hgq1}\end{eqnarray}with $q=[q_1, q_2,\cdots, q_l].$ Combining the series of $q_i$ in \eqref{qi1}, we deduce
\begin{eqnarray*}\|q_i\|_{H^2(\mathbb{D})}^2=\sum_{k=0}^\infty|c_{ki}|^2\;\mbox{for all}\;i\in \{1,\cdots,l\}.\end{eqnarray*} And then the norm estimation in \eqref{cki1} implies
\begin{eqnarray} \|q\|_{H^2(\mathbb{D}, \mathbb{C}^l)}^2=\sum_{i=1}^l \|q_i\|_{H^2(\mathbb{D})}^2=\sum_{i=1}^l\sum_{k=0}^\infty |c_{ki}|^2\leq \|h\|_{1}^2,\label{norm1}\end{eqnarray}
with
$$q=[q_1, q_2,\cdots, q_l]=\sum_{k=0}^\infty B^k C_k\in H^2(\mathbb{D}, \mathbb{C}^l)$$
where $C_k=(c_{k1},\cdots, c_{kl})\in \mathbb{C}^l$.
Then there exists a linear submanifold $$\ma{N}:=\{ q\in H^2(\mathbb{D},\mathbb{C}^l):\; \exists h\in \ma{M},\;h= q G_0\},$$ satisfying  $\ma{M}= \ma{N}G_0$. For all $h\in \ma{M},$ \eqref{norm1} implies
 $$\|h\|_{1}\geq \|q\|_{H^2(\mathbb{D}, \mathbb{C}^l)}.$$

 Next we show $\ma{N}$ is invariant under $T_{\ol{B}}$. Let
$T=T_{B}$ and $\ma{H}=\ma{D}_\alpha$ with $\alpha\in[0,1]$ in Lemma \ref{lem decomp}, and then the operators in \eqref{RQ} become  $$R=(T_{B}^* T_{B})^{-1} T_{B}^{*} P^{\ma{M}\cap T_{B} \ma{D}_\alpha}\;\;\mbox{and}\;\;Q=P^{\ma{M}\ominus(\ma{M}\cap T_{B} \ma{D}_\alpha)}.$$
Hence the equation \eqref{hfor} with $p=0$ gives  $h=Qh+T_{B}Rh$, which together with \eqref{hgq1} entail
\begin{eqnarray*}q G_0 =Q(q G_0)+T_{B} R(q G_0)= C_0 G_0 + B R(q G_0).\end{eqnarray*} And then we obtain
\begin{eqnarray*} BR(q G_0)=(q-C_0) G_0= \left( \sum_{k=1}^\infty B^k C_k\right)G_0,
\end{eqnarray*}
which implies
\begin{eqnarray*}R(q G_0)&=&\left(\sum_{k=1}^\infty B^{k-1} C_k\right)G_0\nonumber\\&=&\left(T_{\ol{B}}\left(\sum_{k=0}^\infty B^{k}C_k\right)\right)G_0\\&=& (T_{\ol{B}} (q))G_0.\end{eqnarray*}
The formula $T_B Rh=P^{\ma{M}\cap T_B \ma{D}_\alpha} h \in \ma{M}$ together with
the fact that
$\ma{M}$ is a nearly $T_B^{-1}$ invariant subspace imply $R(\ma{M})\subset\ma{M}$. In particular, $R(q G_0)\in \ma{M}$  and then $T_{\ol{B}} (q)\in \ma{N}$ from the definition of $\ma{N}.$ This means $\ma{N}$ is a $T_{\ol{B}}$ invariant submanifold of $H^2(\mathbb{D},\mathbb{C}^l)$.
\end{proof}

 Now consider the following special case of \eqref{commute}:

 \begin{eqnarray}\begin{CD}
H^2(\mathbb{D},\mathbb{C}^l) @>T_B>> H^2(\mathbb{D},\mathbb{C}^l)\\
@VV U V @VV U V\\
H^2(\mathbb{D}, \mathbb{C}^{ml}) @>S>> H^2(\mathbb{D}, \mathbb{C}^{ml}).
\end{CD}\label{eq:commuteB}\end{eqnarray}

Here $SU=UT_{B}$ holds for the unilateral shift $S: H^2(\mathbb{D},\mathbb{C}^{ml})\rightarrow H^2(\mathbb{D},\mathbb{C}^{ml})$ and $T_B: H^2(\mathbb{D},\mathbb{C}^l)\rightarrow H^2(\mathbb{D},\mathbb{C}^l)$ with multiplicity $ml$. This leads to the following remark for finite-dimensional nearly $T_{B}^{-1}$ invariant subspaces in $\ma{D}_\alpha$ with $\alpha\in [0,1].$

\begin{rem}\label{remN} In Theorem \ref{thm Dalpha}, if $\ma{M}$ is finite-dimensional, then $\ma{N}\subset H^2(\mathbb{D},\mathbb{C}^l)$ is also finite-dimensional and hence closed.
From the Beurling-Lax Theorem and the commutative diagram \eqref{eq:commuteB}, we deduce that
\[
\ma{N}=U^*(H^2(\mathbb{D},\mathbb{C}^{ml})\ominus\Psi H^2(\mathbb{D},\mathbb{C}^{r}))
\]
 with   $0 \le r\leq ml$  and an inner function $\Psi\in H^\infty(\mathbb{D}, \ma{L}(\mathbb{C}^r, \mathbb{C}^{ml})).$
 Then $\ma{M}=[U^*(H^2(\mathbb{D},\mathbb{C}^{ml})\ominus  \Psi H^2(\mathbb{D},\mathbb{C}^{r}))]G_0.$
\end{rem}

\subsection{Nearly $T_{B}^{-1}$ invariant subspaces in $\ma{D}_\alpha$ with $\alpha\in [-1,0)$}

In this subsection, we cannot make $T_B$ into an expansive operator, but it is possible to achieve a good enough lower bound by taking an
equivalent norm.
So we endow $\ma{D}_\alpha$ with the modified equivalent norm denoted by $\|\cdot\|_2$ as follows:
for any  $f=\sum_{k=0}^\infty B^k h_k$ with $h_k\in K_{B}$,
\begin{eqnarray}\|f\|_{2}^2:=\sum_{k=0}^{N-1}N^\alpha\|h_k\|_{H^2}^2 +\sum_{k=N}^\infty(k+1)^\alpha \|h_k\|_{H^2}^2 ,
\label{eq:d-10}
\end{eqnarray}  where $N$ is a fixed and sufficiently large   positive integer, to be specified below. With respect to the norm $\|\cdot\|_2$,  the lower bound of $T_{B}$ defined in \eqref{gamma} is
\begin{eqnarray}\gamma_{2}:=\left(1-\frac{1}{N+1}\right)^{-\alpha/2}.
\label{gamma2}\end{eqnarray}
 Then it holds that $\|T_Bf\|_{2}^2=\|B f\|_{2}^2 \geq \gamma_2^2 \|f\|_{2}^2\;\mbox{for any}\; f\in \ma{D}_\alpha,$ implying
 that the operator $T:=\gamma_2^{-1} T_{B}: \ma{D}_\alpha\rightarrow  \ma{D}_\alpha$ satisfies
 \begin{eqnarray*}
 \|Tf\|_{2}=\|\gamma_2^{-1} T_{B}f\|_{2}\geq \|f\|_{2}\;\;\mbox{for any}\;f\in \mathcal{D}_\alpha.
 \end{eqnarray*}

So $(\ma{D}_\alpha, T_{B})$ also satisfies \emph{Conditions} $(i)$-$(iv)$ with the lower bound $\gamma_2$  given in \eqref{gamma2} for $\alpha\in[-1,0)$. Here we choose $N$ large enough such that $\gamma_{2}$ satisfies $B^{-1}(D(0,\gamma_2))\supset s \mathbb{D}$ with $s \mathbb{D}$ a disc   containing all the zeros of $B$. This ensures that
\begin{eqnarray}\|\gamma_2^{-1} B\|_{H^\infty(s\mathbb{D})}<1.\label{sgamma} \end{eqnarray}
 Furthermore, $T:=\gamma_2^{-1} T_{B}$ satisfies the assumptions in Lemma \ref{lem decomp} and
 \begin{eqnarray*}\bigcap_{n\in \mathbb{N}} \left(B^n \ma{D}_\alpha\right)|_{s \mathbb{D}}=\bigcap_{n\in \mathbb{N}} \left(T^n \ma{D}_\alpha\right)|_{s \mathbb{D}}=\{0\}.\end{eqnarray*}

Based on the above facts and \cite[Theorem 3.2]{Er}, a lemma similar to Lemma \ref{lem alpha1} holds for the case $\alpha\in [-1,0)$ with $\gamma_2$ in \eqref{gamma2}.

\begin{lem}\label{thm lem} For $\alpha\in[-1,0)$ and  $\gamma_2:=\left(1-\frac{1}{N+1}\right)^{-\alpha/2}$ with large enough $N$, let $\ma{M}$ be a nonzero nearly $T_{B}^{-1}$ invariant subspace of $\ma{D}_\alpha$ endowed with the norm $\|\cdot\|_2$ in \eqref{eq:d-10} and $(g_i)_{i\in\{1,\cdots,l\}}$ be a hilbertian basis of $\ma{M}\ominus (\ma{M}\cap T_{B}\ma{D}_\alpha)$. Then for all $h\in \ma{M},$ there exists $(q_i)_{i\in\{1,\cdots, l\}}$ in $\ma{O}(s\mathbb{D})$ such that \begin{eqnarray}h=\sum_{i=1}^l g_i q_i\;\mbox{on}\;s \mathbb{D},\label{halpha2} \end{eqnarray}
and for all $i\in \{1,\cdots, l\},$ there exists a sequence $\{d_{ki}\}_{k\in \mathbb{N}}\in \mathbb{C}^{\mathbb{N}}$ with
\begin{eqnarray}&&q_i=\sum_{k=0}^\infty d_{ki}\left( \gamma_2^{-1}B \right)^k\;\mbox{on}\;s \mathbb{D},\label{qi2}\\ &&\sum_{i=1}^l\sum_{k=0}^\infty |d_{ki}|^2\leq \|h\|_{2}^2.\label{cki2} \end{eqnarray}\end{lem}
In order to use the submanifold in $H^2(\mathbb{D},\mathbb{C}^l)$ to describe nearly $T_{B}^{-1}$ invariant subspaces in $\ma{D}_\alpha$ with $\alpha\in[-1,0)$, here we introduce an unitary mapping $U_s: H^2(s\mathbb{D},\mathbb{C}^l)\rightarrow H^2(\mathbb{D},\mathbb{C}^l)$ by
\begin{eqnarray*} (U_sf)(z)=f(sz).\end{eqnarray*}
Then the diagram \eqref{commute2}  commutes, with $T_s^*:= U_s T_{B^{-1}} U_s^*$,
 \begin{eqnarray}\begin{CD}
H^2(s\mathbb{D},\mathbb{C}^l) @>T_{B^{-1}}>> H^2(s\mathbb{D},\mathbb{C}^l)\\
@VV U_s V @VV U_s V\\
H^2(\mathbb{D}, \mathbb{C}^l) @>T_s^*>> H^2(\mathbb{D}, \mathbb{C}^l).
\end{CD}\label{commute2}\end{eqnarray}
Since the disc $s\mathbb{D}$ contains all zeros of $B$, the symbol $B^{-1}$
lies in $L^\infty(s\mathbb{T}),$ and thus
 \begin{eqnarray}(T_s^*f)(z)&=&(U_s T_{B^{-1}} U_s^*f)(z)\nonumber\\&=&(U_s T_{B^{-1}})f(s^{-1}z)
 \nonumber\\&=&U_s  \left[P_{H^2(s\mathbb{D}, \mathbb{C}^l)}\left( \frac{1}{B(z)} f(s^{-1}z)\right)\right]\nonumber\\&=&P_{H^2(\mathbb{D}, \mathbb{C}^l)}\left( \frac{1}{B(sz)} f(z)\right)\nonumber\\&=&T_{\frac{1}{B(sz)}} f(z)
 =T_{\ol{B(s^{-1}z)}} f(z),\label{TsB*}\end{eqnarray} due to
 the fact $ B^{-1}(sz)=\ol{B(s^{-1}z)}$ for $z\in \mathbb{T}.$

Based on the above notations, we present a theorem for nearly $T_{B}^{-1}$ invariant subspaces in $\ma{D}_\alpha$ with $\alpha\in[-1,0)$ and $\gamma_2$ in \eqref{gamma2}.

 \begin{thm}\label{thm alpha2} For $\alpha\in[-1,0)$ and  $\gamma_2:=\left(1-\frac{1}{N+1}\right)^{-\alpha/2}$ with large enough $N$, let $\ma{M}$ be a nonzero nearly $T_{B}^{-1}$ invariant subspace of $\ma{D}_\alpha$ endowed with the norm $\|\cdot\|_2$ in \eqref{eq:d-10} and $G_0:=[g_1, g_2,\cdots, g_l]^t$ be a matrix containing an orthonormal basis $(g_i)_{i\in\{1,\cdots,l\}}$ of $\ma{M}\ominus (\ma{M}\cap T_{B}\ma{D}_\alpha)$. Then there exists a linear submanifold $\ma{N}\subset H^2(s\mathbb{D}, \mathbb{C}^l)$ such that $\ma{M}= \ma{N} G_0$ on $s\mathbb{D}$ and for all $h\in \ma{M}$ there exists $q\in \ma{N}$ such that $h=qG_0$ on $s\mathbb{D}$ with $$\|h\|_{2} \geq  \left( 1- \left\|\gamma_2^{-1}B\right\|_{H^\infty(s\mathbb{D})}^2\right)^{1/2} \|q\|_{H^2(s\mathbb{D},\mathbb{C}^l)}.$$
Moreover,  $\ma{N}$ is invariant under $T_{B^{-1}}$ and then $U_s(\ma{N})$ is invariant under $T_s^*:=U_sT_{B^{-1}} U_s^*$ in $H^2(\mathbb{D}, \mathbb{C}^l)$.\end{thm}

\begin{proof} For $\alpha\in[-1, 0)$, the equation \eqref{halpha2} implies
\begin{eqnarray}h=\sum_{i=1}^l g_i q_i=q G_0 \;\mbox{on}\;s\mathbb{D}\label{hgq}\end{eqnarray} with $q=[q_1, q_2,\cdots, q_l].$ The display \eqref{sgamma} and $q_i$ in \eqref{qi2} entail
\begin{eqnarray*}\|q_i\|_{H^2(s\mathbb{D})}&=&\left\|\sum_{k=0}^\infty d_{ki}\left(\gamma_2^{-1}B\right)^k \right\|_{H^2(s\mathbb{D})} \\&\leq& \sum_{k=0}^\infty |d_{ki}|\left\|\gamma_2^{-1}B\right\|_{H^\infty(s\mathbb{D})}^k\\&\leq
&\left(\sum_{k=0}^\infty\left\|\gamma_2^{-1}B\right\|_
{H^\infty(s\mathbb{D})}^{2k}\right)
^{1/2} \left(\sum_{k=0}^\infty |d_{ki}|^2\right)^{1/2} \\&\leq & \left( 1-  \left\|\gamma_2^{-1}B\right\|_{H^\infty(s\mathbb{D})}^2\right)^{-1/2}
\left(\sum_{k=0}^\infty |d_{ki}|^2\right)^{1/2}, \end{eqnarray*} $\;\mbox{for all}\;i=1,\cdots,l$. Thus the above estimations and \eqref{cki2} imply
\begin{eqnarray} \|q\|_{H^2(s\mathbb{D}, \mathbb{C}^l)}^2&=&\sum_{i=1}^l \|q_i\|_{H^2(s\mathbb{D})}^2\nonumber\\ &\leq & \left( 1- \left\|\gamma_2^{-1}B\right\|_{H^\infty(s\mathbb{D})}^2\right)^{-1} \sum_{i=1}^l\sum_{k=0}^\infty |d_{ki}|^2\nonumber\\&\leq& \left( 1- \left\|\gamma_2^{-1}B\right\|_{H^\infty(s\mathbb{D})}^2\right)^{-1} \|h\|_{2}^2<+\infty.\label{qq}\end{eqnarray} This implies $$q=\sum_{k=0}^\infty \left(\gamma_2^{-1}B\right)^k D_k\in H^2(s\mathbb{D},\mathbb{C}^l),$$ where  $D_k=(d_{k1}, d_{k2},\cdots, d_{kl})\in \mathbb{C}^l$. Hence define a linear submanifold
\begin{eqnarray}\ma{N}:=\{q\in H^2(s\mathbb{D}, \mathbb{C}^l):\;\exists h\in \ma{M},\;\; h=qG_0 \;\mbox{on}\;s\mathbb{D}\},\label{n2}\end{eqnarray} satisfying
 $\ma{M}=\ma{N}G_0$ on $s\mathbb{D}$. For all $h\in \ma{M},$ \eqref{qq} gives
 $$\|h\|_{2} \geq  \left( 1- \left\|\gamma_2^{-1}B\right\|_{H^\infty(s\mathbb{D})}^2\right)^{1/2} \|q\|_{H^2(s\mathbb{D},\mathbb{C}^l)}.$$

 Next we show $\ma{N}$ is invariant under $T_{B^{-1}}$. Let $T=\gamma_2^{-1} T_{B}$ in the display \eqref{RQ}, and then the equation \eqref{hfor} with $p=0$ gives
 \begin{eqnarray*}h=Qh+ TR h=Q h+\gamma_2^{-1} T_{B}Rh. \end{eqnarray*}
 On $s\mathbb{D},$ the above equation together with \eqref{hgq} entail
 \begin{eqnarray*}q G_0=Q(q G_0)+\gamma^{-1}_2 T_{B} R(q G_0)=D_0 G_0+ \gamma_2^{-1} B R(q G_0),\end{eqnarray*} which further verifies
\begin{eqnarray*} \gamma_2^{-1}B R(q G_0)&=&(q-D_0) G_0
\\&=& \left( \sum_{k=1}^\infty \left(\gamma_2^{-1}B\right)^k D_k\right) G_0\;\mbox{on}\;s\mathbb{D}.\end{eqnarray*}
Letting $B^{-1}$ act on both sides, it yields that
\begin{eqnarray*}\gamma_2^{-1}R(q G_0)&=& \gamma_2^{-1} \left(\sum_{k=1}^\infty \left(\gamma_2^{-1}B\right)^{k-1}D_k \right)G_0\\&=&\gamma_2^{-1}\left(\gamma_2 T_{B^{-1}}(q)\right)G_0
\\& =& \left(T_{B^{-1}}(q)\right)G_0,
\end{eqnarray*}
which together with $\gamma_2^{-1}R(q G_0)\in \ma{M}$ entail $ T_{B^{-1}} (q)\in \ma{N}$ from the definition of $\ma{N}$ in \eqref{n2}. That means $\ma{N}$ is $T_{B^{-1}}$ invariant in $H^2(s\mathbb{D},\mathbb{C}^l).$  Finally,  the commutative diagram \eqref{commute2} ensures that $T_s^*(U_s(\ma{N}))\subset U_s(\ma{N}),$ ending the proof. \end{proof}

 In order to illustrate the operator $T_s^*$ given in  \eqref{TsB*}, we firstly take a degree-$1$ Blaschke product $B(z)=(a-z)(1-\ol{a}z)^{-1}$ with $a\in \mathbb{D},$  it is easy to obtain
\begin{eqnarray}  B(s^{-1}z)=\frac{a-s^{-1}z}{1-\ol{a}s^{-1}z}=\frac{1}{s}
\frac{as-z}{1-\ol{a}sz}\frac{1-\ol{a}sz}{1-\ol{a}s^{-1}z}.\label{BSZ}\end{eqnarray} Since the disc $s\mathbb{D}$ contains the zero of $B$, that means $|\ol{a}s^{-1}|<1,$ and then the last term in \eqref{BSZ} is an invertible analytic function on $\mathbb{D}.$ This is to say
that $B(s^{-1}z)$ can be written as a degree-$1$ Blaschke product times an invertible analytic function.

Generally, if $B$ is a degree-$m$ Blaschke product, it can be similarly calculated that
\begin{eqnarray*}B(s^{-1}z)=b(z) F_s(z)\label{BSF}\end{eqnarray*} with a degree-$m$ Blaschke product $b(z)$ and an invertible analytic function $F_s(z)$ on $\mathbb{D}$. So $T_s^*=T_{\ol{b F_s}}$ and then $T_s =T_{b F_s}$ on $H^2(\mathbb{D},\mathbb{C}^l).$

Suppose $\ma{F}\subset H^2(\mathbb{D},\mathbb{C}^l)$ is a $T_{b}$ invariant subspace, Theorem \ref{thm A} implies $\ma{F}=A H^2(\mathbb{D},\mathbb{C}^l)$ with some $T_{b}$-inner operator $A$. For the special case $l=1,$ it follows that $\ma{F}=\theta H^2(\mathbb{D})$ with an inner function $\theta$. The fact $F_s H^2(\mathbb{D})=H^2(\mathbb{D})$ entails $\theta H^2(\mathbb{D})$ is also a $T_{bF_s}$-invariant subspace in $H^2(\mathbb{D})$. So the model space $K_\theta$ is $T_s^*$ invariant.

In general, there is no simple description for $T^*_s$-invariant subspaces of $H^2(\DD)$, although in the finite-dimensional case, elementary linear algebra tells us that they are spanned by generalized eigenvectors of $T^*_s$, that is, elements of Toeplitz kernels. The paper \cite{CaP} is relevant here.





\subsection*{Acknowledgments.}
 This work was done while Yuxia Liang was visiting the University of Leeds. She is grateful to the School of Mathematics at the University of Leeds for its warm hospitality. Yuxia Liang is supported  by the National Natural Science Foundation of China (Grant No. 11701422).

\subsection*{Data availability.}
All data generated or analysed during this study are included in this published article.
 

\begin{thebibliography}{}

\bibitem{CaP1}C\^{a}mara, M.C.,  Partington, J.R.: Near invariance and kernels of Toeplitz operators. Journal d'Analyse  Math. 124, 235--260 (2014)

\bibitem{CaP}  C\^{a}mara, M.C.,  Partington, J.R.: Finite-dimensional Toeplitz kernels and nearly-invariant subspaces. J. Operator Theory. 75(1), 75--90 (2016)

\bibitem{CCP} Chalendar, I.,  Chevrot, N.,  Partington, J.R.: Nearly invariant subspaces for backwards shifts on vector-valued Hardy spaces. J. Operator Theory. 63(2), 403--415 (2010)

\bibitem{CGP}  Chalendar, I.,  Gallardo-Guti\'errez, E.A., Partington, J.R.:  Weighted composition operators on the Dirichlet space: boundedness and spectral properties. Math. Ann. 363, 1265--1279 (2015)

\bibitem{GPS}  Gallardo-Guti\'{e}rrez, E.A.,  Partington J.R.,  Seco, D.: On the wandering property in Dirichlet spaces. Integral Equations and Operator Theory. 92, no. 2, paper no. 16 (2020)

\bibitem{EKMR} El-Fallah, O., Kellay, K.,  Mashreghi J., Ransford, T.: A Primer on the Dirichlet Space. Cambridge Tracts in Mathematics, 203, Cambridge University Press, Cambridge (2014)

\bibitem{Er}  Erard, C.: Nearly invariant subspaces related to multiplication operators in Hilbert spaces of analytic functions. Integral Equations and Operator Theory. 50, 197--210 (2004)

\bibitem{GMR}  Garcia, S.,  Mashreghi, J.,  Ross, W.T.: Introduction to Model Spaces and Their Operators. Cambridge: Cambridge University Press (2016)

\bibitem{Ha}  Hayashi, E.: The kernel of a  Toeplitz operator. Integral Equations and Operator Theory. 9(4), 588--591 (1986)


\bibitem{hitt}  Hitt, D.: Invariant subspaces of ${\ma H}^2$ of an annulus. Pacific J. Math. 134(1), 101--120 (1988)

\bibitem{LS}  Lance, T.L.,  Stessin, M.I.: Multiplication invariant subspaces of Hardy spaces. Can. J. Math. 49(1), 100--118 (1997)

\bibitem{LP} Liang, Y.,  Partington, J.R.: Representing kernels of perturbations of Toeplitz operators by backward shift-invariant subspaces. Integral Equations and Operator Theory. 92, no. 4, paper no. 35 (2020)

\bibitem{RR}  Rosenblum, M.,  Rovnyak, J.: Hardy classes and operator theory. Oxford University Press, New York (1985)

\bibitem{Sa1} Sarason, D.: Nearly invariant subspaces of the backward shift, Contributions to operator theory and its applications (Mesa, AZ, 1987), 481-493, Oper. Theory Adv. Appl. 35, Birkh\"{a}user, Basel (1988)

\bibitem{Sa2} Sarason, D.: Kernels of Toeplitz operators. Oper. Theory: Adv. Appl. 71, 153--164 (1994)

\end{thebibliography}
\end{document}